\documentclass[a4paper,11pt]{article}
\usepackage[margin=2cm]{geometry}
\usepackage[T1]{fontenc}
\usepackage[utf8]{inputenc}
\usepackage{amsmath,amssymb,amsthm}
\usepackage{graphicx}
\usepackage{xcolor}
\usepackage{tikz}
\usepackage[hidelinks]{hyperref}

\theoremstyle{plain}
\newtheorem{theorem}{Theorem}

\newtheorem{corollary}[theorem]{Corollary}

\newtheorem{claim}[theorem]{Claim}
\newtheorem{conjecture}[theorem]{Conjecture}

\theoremstyle{definition}

\newtheorem{problem}[theorem]{Problem}

\theoremstyle{remark}

\theoremstyle{plain}

\newcommand{\calA}{\mathcal{A}}

\newcommand{\gL}{\gamma^{\rm LD}}

\let\oldenumerate\enumerate
\renewcommand{\enumerate}{
  \oldenumerate
  \setlength{\itemsep}{1pt}
  \setlength{\parskip}{0pt}
  \setlength{\parsep}{0pt}
}

\def\vertex(#1){\put(#1){\circle*{2}}}
\def\vertexo(#1){\put(#1){\circle{2}}}
\def\vert(#1){\put(#1){\circle*{1.5}}}
\def\verto(#1){\put(#1){\circle{1.5}}}
\def\lab(#1)#2{\put(#1){\makebox(0,0)[c]{#2}}}
\title{A note on partitioning the vertex set of a graph into a dominating set and a locating dominating set}

\author{%
\begin{minipage}{0.96\textwidth}
\centering
Dipayan Chakraborty\textsuperscript{1,2,3,4}\quad
Florent Foucaud\textsuperscript{1}\\[0.4em]
Michael A. Henning\textsuperscript{2}\quad
Tero Laihonen\textsuperscript{5}\\[1em]
\small
\textsuperscript{1}Université Clermont Auvergne, CNRS, Clermont Auvergne INP,
Mines Saint-Étienne, LIMOS, 63000 Clermont-Ferrand, France.\\
\href{mailto:dipayancha@gmail.com}{\nolinkurl{dipayancha@gmail.com}},
\href{mailto:florent.foucaud@uca.fr}{\nolinkurl{florent.foucaud@uca.fr}}\\[0.6em]
\textsuperscript{2}Department of Mathematics and Applied Mathematics,
University of Johannesburg, South Africa.\\
\href{mailto:mahenning@uj.ac.za}{\nolinkurl{mahenning@uj.ac.za}}\\[0.6em]
\textsuperscript{3}Department of Computer Science and Mathematics,
Lebanese American University, Beirut, Lebanon.\\[0.6em]
\textsuperscript{4}Centrale Méditerranée, Laboratoire d'Informatique et Systèmes
UMR 7020, Marseille, France.\\[0.6em]
\textsuperscript{5}Department of Mathematics and Statistics,
University of Turku, Turku, FI-20014, Finland.\\
\href{mailto:terolai@utu.fi}{\nolinkurl{terolai@utu.fi}}
\end{minipage}%
}
\date{}

\begin{document}

\maketitle

\begin{abstract}
A set $S$ of vertices in a graph $G$ is a dominating set of $G$ if every vertex not in $S$ has a neighbor in $S$, where two vertices are neighbors if they are adjacent. The domination number, $\gamma(G)$, of $G$ is the minimum cardinality among all dominating sets of $G$. Given a set $S$ of vertices of a graph $G$, two vertices are located by $S$ if they have distinct sets of neighbors in $S$. Moreover, if $S$ locates every pair of vertices not in $S$, then it is called a locating set of $G$. A locating dominating set of $G$ is both a dominating and a locating set of $G$. The locating domination number, $\gL(G)$, is the minimum cardinality among all locating dominating sets of $G$. A notable conjecture in the study of locating dominating sets is to show that the locating domination number of an isolate-free and twin-free graph $G$ of order $n$ is at most $\frac{1}{2}n$. So far, the best approximation to this $\frac{1}{2}n$-upper bound conjecture is known to be $\left \lceil \frac{5}{8}n \right \rceil$. In fact, much in line with the conjecture, an even stronger reformulation proposed in the literature is to ask if it is possible to partition the vertex set of an isolate-free and twin-free graph into two locating sets. However, such partitions into locating sets may not exist if the graph is also allowed to have twins. Continuing with this line of research, we show that if $G$ is an isolate-free (and not necessarily twin-free) graph, then the vertex set of $G$ can be  partitioned into a dominating set and a locating dominating set. As a consequence, we infer that every isolate-free graph $G$ of order~$n$ satisfies $\gamma(G) + \gL(G) \le n$, and we show that the last bound is tight. Moreover, our proof of the existence of a partition of the vertex set of an isolate-free graph into a dominating and a locating dominating set also provides a polynomial-time algorithm to construct such a partition.
\end{abstract}

\section{Introduction}

Our aim is to show that if $G$ is an isolate-free graph, then the vertex set of $G$ can be  partitioned into a dominating set and a locating dominating set. We begin with some basic definitions. Let $G$ be a graph with vertex set $V(G)$, edge set $E(G)$, and of \emph{order} $n = |V(G)|$ and \emph{size} $m(G)=|E(G)|$. Two vertices $u$ and $v$ of $G$ are \emph{adjacent} if $uv \in E(G)$. Two adjacent vertices are called \emph{neighbors}. The \emph{open neighborhood} $N_G(v)$ of a vertex $v$ in $G$ is the set of neighbors of $v$, while the \emph{closed neighborhood} of $v$ is the set $N_G[v] = \{v\} \cup N_G(v)$. For a set $S\subseteq V(G)$, its \emph{open neighborhood} is the set $N_G(S)=\cup_{v\in S} N_G(v)$, and its \emph{closed neighborhood} is the set $N_G[S] = N_G(S)\cup S$. Two vertices $u$ and $v$ are \emph{closed twins} in $G$ if they have the same closed neighborhoods, that is, if $N_G[u] = N_G[v]$. The vertices $u$ and $v$ are \emph{open twins} if they have the same open neighborhoods, that is, if $N_G(u) = N_G(v)$. A graph without closed or open twins is called \emph{twin}-\emph{free}.

The \emph{degree} of a vertex $v$ in $G$ is the number of vertices adjacent to $v$ in $G$, and is denoted by $\deg_G(v)$, and so $\deg_G(v) = |N_G(v)|$. The maximum (minimum) degree among the vertices of $G$ is denoted by $\Delta(G)$ ($\delta(G)$, respectively). An \emph{isolated vertex} in a graph $G$ is a vertex of degree zero. A graph without any isolated vertex is called an \emph{isolate-free} graph. A \emph{leaf} in $G$ is a vertex of degree~$1$ in $G$, and an edge incident with a leaf is a called a \emph{pendant edge}. A \emph{support vertex} is a vertex with at least one leaf neighbor. For a set $S \subseteq V(G)$, the subgraph induced by $S$ is denoted by $G[S]$. Further, the subgraph of $G$ obtained from $G$ by deleting all vertices in $S$ and all edges incident with vertices in $S$ is denoted by $G - S$; that is, $G-S = G[V(G)\setminus S]$. We use $P_n$, $C_n$, and $K_n$ to denote a \emph{path}, a \emph{cycle}, and a \emph{complete graph}, respectively, on $n$ vertices. For $1 \le r \le s$, we denote a \emph{complete bipartite graph} with partite sets of cardinalities $r$ and $s$, respectively, by $K_{r,s}$. A \emph{star} is a complete bipartite graph $K_{1,s}$ for some $s \ge 1$.

A set $S \subseteq V(G)$ is a \emph{dominating set} of a graph $G$ if every vertex in $V(G) \setminus S$ has at least one neighbor in $S$. Equivalently, $S$ is a dominating set of $G$ if $N_G[v] \cap S \ne \emptyset$ for all vertices $v$ of $G$. A vertex $v$ in $G$ \emph{dominates} itself and all its neighbors. The \emph{domination number} $\gamma(G)$ is the minimum cardinality among all dominating sets in $G$.  A \emph{dominating vertex} in $G$ is a vertex $v$ adjacent to every other vertex of $G$.  A \emph{total dominating set} of an isolate-free graph $G$ is a dominating set $S$ with the property that $G[S]$ is isolate-free, that is, $N_G(v) \cap S \ne \emptyset$ for all vertices $v$ of $G$. A dominating set $S$ is an \emph{independent dominating set} of $G$ if $S$ is an independent set in $G$, that is, $G[S]$ consists of isolated vertices.  We refer to the books~\cite{HaHeHe-20,HaHeHe-21,HaHeHe-23,HeYe-book} for a detailed study of domination and its variants.

Given a set $S$ of vertices of a graph $G$, two vertices are \emph{located} by $S$ if they have distinct sets of neighbors in $S$. If $S$ locates each pair of distinct vertices in $V(G) \setminus S$, then $S$ is called a \emph{locating set} of $G$. A \emph{locating dominating set} of $G$, abbreviated LD-set, is a dominating set $S$ of $G$ with the additional property that every two vertices not in $S$ are located by $S$. Hence an LD-set is a vertex subset $S$ of $G$ such that $S$ is a dominating set of $G$ and $N_G(u) \cap S \ne N_G(v) \cap S$ for all distinct vertices $u, v \in V(G) \setminus S$. Thus, every two vertices outside the dominating set $S$ are located by the set $S$. The \emph{locating domination number}, $\gL(G)$, is the minimum cardinality among all LD-sets of $G$. The concept of an LD-set was introduced by Slater in the 1980s~\cite{Sl-88}, who first coined the concept \emph{locating dominating set}. We remark that a locating dominating set is also called a locating dominating code in the literature. We refer the reader to the excellent survey article on locating dominating sets in graphs by Lobstein, Hudry, and Charon~\cite{LoHuIr-20}. Locating dominating sets are now very well studied in the literature (see,~\cite{BoChFa-25,ChFoHaHeWa-24,ChFoLe-24,ChFoPaWa-24} for a small sample of recent papers on the topic).

\subsection{Motivation}

The following classic 1962 result by Ore showed that for any isolate-free graph, its vertex set can be partitioned into two dominating sets.

\begin{theorem}{\rm (\cite{ore})}
\label{ore}
If $G$ is an isolate-free graph of order~$n$ and $S$ is a minimal dominating set of $G$, then $V(G) \setminus S$ is a dominating set, and thus $\gamma(G) \le \frac{1}{2}n$.
\end{theorem}

Ore's result that the vertex set of graph can be partitioned into two dominating sets does not necessarily extend to other types of domination parameters. For example, the vertices of a $5$-cycle cannot be partitioned into a dominating set and a total dominating set. However it is shown in~\cite{HeSo2008} that if $G$ is a connected graph with minimum degree at least~$2$ and $G$ is not the $5$-cycle, then the vertex set of $G$ can be partitioned into a dominating set and a total dominating set.  Hence, a natural problem is to consider which graphs can be partitioned into two specific types of dominating sets.  Such problems have been studied in~\cite{Bange78,brodorgodhenn,DeDeHa2019,DeHaHe2017,dorfhenn2005,GoHe-18,HaHe_24a,HaHe_24b,HeTe-98,HeLoRa2010,HeLoRaSo,HePe-19,HeSo2008,HeSo2009,HeYe_Ch13,Justin-PhD,SoHe2011} and elsewhere.

Much of the interest in locating dominating sets is the $\frac{1}{2}n$-upper bound conjecture on the locating domination number of twin-free and isolate-free graphs by Garijo, Gonz\'alez, and M\'arquez~\cite{GaGpMa-14} in 2014, with a slightly stronger reformulation by Foucaud and Henning in~\cite{FoHe-16b}.

\begin{conjecture}{\rm (\cite{GaGpMa-14,FoHe-16b})}
\label{conj:Ldom-bound}
If $G$ is an isolate-free, twin-free graph of order~$n$, then $\gL(G) \le \frac{1}{2}n$.
\end{conjecture}

Conjecture~\ref{conj:Ldom-bound} is known to hold for special classes of graphs, including graphs without $4$-cycles~\cite{GaGpMa-14}, graphs of girth at least~$5$ and minimum degree at least~$2$~\cite{BFH2015}, split graphs and co-bipartite graphs~\cite{FHLS2016}, line graphs~\cite{FH2017}, maximal outerplanar graphs~\cite{CGHMMT2021}, block graphs~\cite{ChFoPaWa-24} and cubic graphs~\cite{FoHe-16b} with the latter being generalized to subcubic graphs~\cite{CHL2024}. As a motivation to prove Conjecture~\ref{conj:Ldom-bound}, several authors have considered the following problem which has a slightly stronger formulation.

\begin{problem}{\rm (\cite{GaGpMa-14,FoHe-16, FHLS2016})}
\label{problem:disjoint-LDS}
Determine classes of isolate-free, twin-free graphs $G$ with the property that $V(G)$ can be partitioned into two locating dominating sets.
\end{problem}

Notice that determining the graph class in Problem~\ref{problem:disjoint-LDS} to be ``all isolate-free and twin-free graphs'' necessarily proves Conjecture~\ref{conj:Ldom-bound}. Problem~\ref{problem:disjoint-LDS} has so far been answered positively for bipartite graphs~\cite{GaGpMa-14}, for block graphs~\cite{ChFoPaWa-24} and, more recently, in~\cite{FoMaPaPr-26} for distance-hereditary graphs, maximal outerplanar graphs, split graphs and co-bipartite graphs. However, it is worth mentioning here that the partitioning technique of taking a minimal locating dominating set $D$ of an isolate-free and twin-free graph $G$ and showing that its complement $V(G) \setminus D$ is also locating dominating does not work to prove Problem~\ref{problem:disjoint-LDS}, as illustrated in~\cite[Figure~3]{GaGpMa-14}. However, as a preliminary investigation, it is not yet known to the authors if there exist examples of isolate-free and twin-free graphs whose vertex sets cannot be partitioned into two locating dominating sets (notice that finding such an example does not, however, disprove Conjecture~\ref{conj:Ldom-bound}	as long as a minimum locating dominating set of the example has cardinality at most $\frac{1}{2}n$). In a recent paper by Bousquet et al.~\cite{BoChFa-25} where the best-known approximation to the $\frac{1}{2}n$-upper bound conjecture has been improved to $\left \lceil \frac{5}{8}n \right \rceil$, the authors also ask a more general version of the question in Problem~\ref{problem:disjoint-LDS} that effectively translates to asking the following.

\begin{problem}{\rm (\cite{BoChFa-25})}
\label{problem:disjoint-LS}
Given an isolate-free and twin-free graph $G$, can its vertex set be partitioned into two locating sets?
\end{problem}

Notice that any graph which satisfies Problem~\ref{problem:disjoint-LDS} automatically also satisfies Problem~\ref{problem:disjoint-LS}, since a locating dominating set is also, by definition, a locating set (and moreover, one can obtain a locating dominating set from any locating set by adding at most one vertex to the latter set). Thus, any isolate-free and twin-free graph from either the class of block graphs or bipartite graphs satisfies Problem~\ref{problem:disjoint-LS}. Moreover, it is not too difficult to verify that, for example, the vertex set of a graph from the class of isolate-free and twin-free either cobipartite or split graphs can be partitioned into its two canonical parts (which are either cliques or independent sets) each of which serves as a locating set. Notice that answering Problem~\ref{problem:disjoint-LS} will imply $\gL(G) \le \left \lceil \frac{1}{2} (n+1) \right \rceil$ for any general isolate-free and twin-free graph $G$, which too would bring us quite close to proving Conjecture~\ref{conj:Ldom-bound}.

Thus, in this line of research as prescribed in both Problems~\ref{problem:disjoint-LDS} and~\ref{problem:disjoint-LS}, given an isolate-free and twin-free graph $G$, the objective is to maximize the sum-total of the ``locating (dominating) power'' of vertex subset $A \subset V(G)$ on its complement $V(G) \setminus A$ and vice-versa. However, such questions need to be slightly modified when we allow the graphs to contain twins. For example, if $G \cong K_{1,n-1}$ (which contains open twins), then it is certainly not possible to partition the vertex set of $G$ into two locating dominating (or locating) sets since any locating (dominating) set of $G$ requires at least all-but-one of its leaves and thus, its complement cannot be a locating (dominating) set. Similar is the story for $G \cong K_n$ (which contains closed twins). In both these examples, however, it is possible to partition the vertex set into a dominating set (consisting of a single dominating vertex of $G$) and a locating dominating set (the complement of the dominating set). Therefore, this motivates us to consider the following question on isolate-free (and not necessarily twin-free) graphs.

\begin{problem} \label{problem:disjoint-DLD}
Can an isolate-free graph be partitioned into a dominating and a locating dominating set?
\end{problem}

In~\cite{GaGpMa-14}, the authors show that the answer to Problem~\ref{problem:disjoint-DLD} is ``yes'' if the input graph is also twin-free. In fact, they prove that if $D$ is any minimal dominating set of an isolate-free and twin-free graph $G$, then $V(G) \setminus D$ is a locating-dominating set of $G$. However, the proof technique for the last result in~\cite{GaGpMa-14} does not adapt to prove Problem~\ref{problem:disjoint-DLD} if we allow $G$ to have twins as well. This can been seen by taking $G \cong K_{1,n}$ (or $G \cong C_4$), for example, where the set $D$ of all leaves (respectively, two non-adjacent vertices) form a minimal dominating set but the set $V(G) \setminus D$ is not locating dominating.

In this paper, we answer Problem~\ref{problem:disjoint-DLD} positively, that is, without the extra requirement of the input graph being twin-free. Moreover, our proof of the existence of a partition of the vertex set of an isolate-free graph into a dominating and a locating domonating set also provides a polynomial-time algorithm to construct such a partition. This also immediately implies that for any isolate-free graph $G$ of order $n$, we have $\gL(G) \le n - \gamma(G)$. Hence, the partition proposed in Problem~\ref{problem:disjoint-DLD} may be considered as a generalization of the $\frac{1}{2}n$-upper bound in Conjecture~\ref{conj:Ldom-bound} for isolate-free graphs which are allowed to have twins. Moreover, this $(n - \gamma(G))$-upper bound is tight since, considering the same examples as before, for any graph $G \cong K_{1,n-1}$ or $G \cong K_n$, we have $\gamma(G) = 1$ and $\gL(G) = n-1$.

A similar result for \emph{identifying codes}\footnote{an \emph{identifying code} $D$ of a graph $G$ is one which is a dominating set of $G$ and has the extra condition that $N_G[u] \cap D \ne N_G[v] \cap D$ for all distinct $u,v \in V(G)$. Identifying codes exist only when the graph has no closed twins. The \emph{identification number} of a graph $G$, denoted $\gamma^{\rm ID} (G)$, is the smallest cardinality among all identifying codes of $G$} has been established by Chakraborty et al. in~\cite{CFHL_Trees2024}, where it has been shown that $\gamma(T) + \gamma^{\rm ID}(T) \le n$ for all trees of order $n$ except when $T \cong P_4$. Since it is not too difficult to see that an identifying code of a graph $G$ is also a locating dominating set of $G$, it can be inferred that $\gamma(T) + \gL(T) \le n$ as well for all trees (also when $T \cong P_4$ by simply considering one set containing the two leaves and the other set being its complement). However, our work in this paper proves that the last inquality holds for all isolate-free graphs.

\section{Main results}

In this paper we study graphs whose vertex set can be partitioned into a dominating set and a locating dominating set. We refer to such a partition of the vertices of a graph $G$ as a DLDS-\emph{partition} of $G$ (standing for ``dominating, locating dominating set partition''). If $G$ has a DLDS-partition, then we say that $G$ is a DLDS-\emph{graph}, standing for ``dominating, locating dominating set graph.'' We shall prove the following result which shows that the vertex set of every isolate-free graph has a DLDS-partition.

\begin{theorem}
\label{thm:main}
Every isolate-free graph is a DLDS-graph.
\end{theorem}

As an immediate consequence of Theorem~\ref{thm:main}, we infer the following result.

\begin{corollary}
\label{cor:main}
If $G$ is an isolate-free graph of order~$n$, then $\gamma(G) + \gL(G) \le n$.
\end{corollary}

We again remark that the bound in Corollary~\ref{cor:main} is tight. Apart from stars $K_{1,n-1}$ and complete graphs $K_n$, another simple example is to take the corona of a graph, where the \emph{corona} $G = H \circ K_1$ of a connected graph $H$ is the graph obtained from $H$ by adding for each vertex $v \in V(H)$ a new vertex $v'$ and the edge $vv'$. In this case, if $D$ is the set of all vertices added to $H$ when constructing $G$, then the partition $\{D,V(H)\}$ is an example of a DLDS-partition of $V(G)$ into a dominating set $D$ and a locating dominating set $V(H)$. In this example, if $G$ has order~$n$, then $\gamma(G) = \gL(G) = \frac{1}{2}n$, and so  $\gamma(G) + \gL(G) = n$. For example, the corona $G = C_5 \circ K_1$ of a $5$-cycle is illustrated in Figure~\ref{fig:corona}, where the set $D$ is indicated by the white vertices and the locating dominating set by the shaded vertices. (We remark that the DLDS-partition in this example is not unique.)

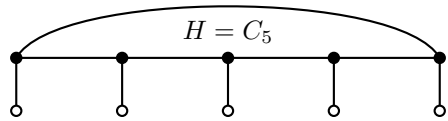
\begin{figure}[htb]
\begin{center}
\begin{tikzpicture}[scale=.7,style=thick,x=1cm,y=1cm]
\def\vr{2.75pt}
\path (0,3) coordinate (v4);
\path (0,4) coordinate (v5);
\draw (v4) -- (v5);
\path (2,3) coordinate (u4);
\path (2,4) coordinate (u5);
\draw (u4) -- (u5);
\path (4,3) coordinate (w4);
\path (4,4) coordinate (w5);
\draw (w4) -- (w5);
\path (6,3) coordinate (x4);
\path (6,4) coordinate (x5);
\draw (x4) -- (x5);
\path (8,3) coordinate (y4);
\path (8,4) coordinate (y5);
\draw (y4) -- (y5);
\draw (v5) -- (u5);
\draw (u5) -- (w5);
\draw (w5) -- (x5);
\draw (x5) -- (y5);
\draw (v5) to[out=60,in=120, distance=1.5cm ] (y5);
\draw (v4) [fill=white] circle (\vr);
\draw (v5) [fill=black] circle (\vr);
\draw (u4) [fill=white] circle (\vr);
\draw (u5) [fill=black] circle (\vr);
\draw (w4) [fill=white] circle (\vr);
\draw (w5) [fill=black] circle (\vr);
\draw (x4) [fill=white] circle (\vr);
\draw (x5) [fill=black] circle (\vr);
\draw (y4) [fill=white] circle (\vr);
\draw (y5) [fill=black] circle (\vr);
%
\draw (4,4.5) node {{\small $H = C_5$}};
\end{tikzpicture}
\caption{A DLDS-partition of the corona $G = C_5 \circ K_1$ of a $5$-cycle}
\label{fig:corona}
\end{center}
\end{figure}

\section{Proof of Theorem~\ref{thm:main}}
\label{S:main-proof}

We proceed by induction on the order~$n$ of an isolate-free graph $G$ to show that $G$ is a DLDS-graph. If $n = 2$, then $G = K_2$ and trivially $G$ has a DLDS-partition (where we add one of the vertices to the dominating set $D$ and the other vertex to the locating dominating set $L$, as illustrated in Figure~\ref{fig:n3n4}(a)). If $n = 3$, then either $G = P_3$ or $G = K_3$. In both cases, the graph $G$ is a DLDS-graph, as illustrated in Figure~\ref{fig:n3n4}(b) and~\ref{fig:n3n4}(c) where the shaded vertices form a locating dominating set $L$ and the white vertices a dominating set $D$ resulting in a DLDS-partition $\{D,L\}$. This establishes the base cases. Let $n \ge 4$ and assume that every isolate-free graph of order less than~$n$ is a DLDS-graph.

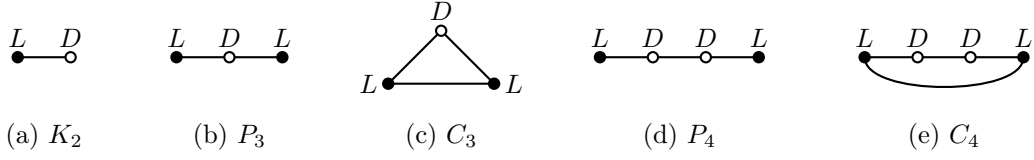
\begin{figure}[htb]
\begin{center}
\begin{tikzpicture}[scale=.7,style=thick,x=1cm,y=1cm]
\def\vr{2.75pt}
\path (-3,0.5) coordinate (t1);
\path (-2,0.5) coordinate (t2);
\draw (t1)--(t2);
\path (0,0.5) coordinate (u1);
\path (1,0.5) coordinate (u2);
\path (2,0.5) coordinate (u3);
\draw (u1)--(u2)--(u3);
\path (4,0) coordinate (v1);
\path (5,1) coordinate (v2);
\path (6,0) coordinate (v3);
\draw (v1)--(v2)--(v3)--(v1);
\path (8,0.5) coordinate (w1);
\path (9,0.5) coordinate (w2);
\path (10,0.5) coordinate (w3);
\path (11,0.5) coordinate (w4);
\draw (w1)--(w2)--(w3)--(w4);
\path (13,0.5) coordinate (x1);
\path (14,0.5) coordinate (x2);
\path (15,0.5) coordinate (x3);
\path (16,0.5) coordinate (x4);
\draw (x1)--(x2)--(x3)--(x4);
\draw (x1) to[out=270,in=270, distance=0.75cm ] (x4);
%
\draw (t1) [fill=black] circle (\vr);
\draw (t2) [fill=white] circle (\vr);
\draw (u1) [fill=black] circle (\vr);
\draw (u2) [fill=white] circle (\vr);
\draw (u3) [fill=black] circle (\vr);
\draw (v1) [fill=black] circle (\vr);
\draw (v2) [fill=white] circle (\vr);
\draw (v3) [fill=black] circle (\vr);
\draw (w1) [fill=black] circle (\vr);
\draw (w2) [fill=white] circle (\vr);
\draw (w3) [fill=white] circle (\vr);
\draw (w4) [fill=black] circle (\vr);
\draw (x1) [fill=black] circle (\vr);
\draw (x2) [fill=white] circle (\vr);
\draw (x3) [fill=white] circle (\vr);
\draw (x4) [fill=black] circle (\vr);
%
\draw (-2.5,-1) node {{\small (a) $K_2$}};
\draw (1,-1) node {{\small (b) $P_3$}};
\draw (5,-1) node {{\small (c) $C_3$}};
\draw (9.5,-1) node {{\small (d) $P_4$}};
\draw (14.5,-1) node {{\small (e) $C_4$}};
\draw[anchor = south] (t1) node {{\small $L$}};
\draw[anchor = south] (t2) node {{\small $D$}};
\draw[anchor = south] (u1) node {{\small $L$}};
\draw[anchor = south] (u2) node {{\small $D$}};
\draw[anchor = south] (u3) node {{\small $L$}};
\draw[anchor = east] (v1) node {{\small $L$}};
\draw[anchor = south] (v2) node {{\small $D$}};
\draw[anchor = west] (v3) node {{\small $L$}};
\draw[anchor = south] (w1) node {{\small $L$}};
\draw[anchor = south] (w2) node {{\small $D$}};
\draw[anchor = south] (w3) node {{\small $D$}};
\draw[anchor = south] (w4) node {{\small $L$}};
\draw[anchor = south] (x1) node {{\small $L$}};
\draw[anchor = south] (x2) node {{\small $D$}};
\draw[anchor = south] (x3) node {{\small $D$}};
\draw[anchor = south] (x4) node {{\small $L$}};
\end{tikzpicture}
\caption{DLDS-partitions of graphs of small order}
\label{fig:n3n4}
\end{center}
\end{figure}

We note that an isolate-free graph is a DLDS-graph if and only if every component of the graph is a DLDS-graph. If $G$ is disconnected, then applying the inductive hypothesis to every component of $G$ we infer that $G$ is a DLDS-graph. Hence we may assume that $G$ is connected. We proceed further with the following series of claims that establish structural properties of the graph $G$.

\begin{claim}
\label{claim0}
If $G$ has a dominating vertex, then $G$ is a DLDS-graph.
\end{claim}

\begin{proof}
Suppose that $G$ has a dominating vertex~$v$, and so $\deg_G(v) = n-1$. In this case, we let $D = \{v\}$ be the singleton set consisting of the dominating vertex~$v$ and we let $L = V(G) \setminus \{v\}$. The resulting set $\{D,L\}$ is a DLDS-partition of $G$, implying that $G$ is a DLDS-graph.
\end{proof}

By Claim~\ref{claim0}, we may assume that $G$ has no dominating vertex, for otherwise $G$ is a DLDS-graph as desired. Thus, $\Delta(G) \le n-2$. With this assumption, if $n = 4$, then $G = P_4$ or $G = C_4$. In both cases, we let $D$ be a set consisting of two adjacent vertices of degree~$2$ and we let $L$ denote the remaining two vertices of $G$ as illustrated in Figure~\ref{fig:n3n4}(d) and~\ref{fig:n3n4}(e). The resulting set $\{D,L\}$ is a DLDS-partition of $G$. (We remark that the DLDS-partition in the above examples are not necessarily unique.) Hence, we may assume that $n \ge 5$, for otherwise the desired result is immediate.

\begin{claim}
\label{claim1}
If $\delta(G) = 1$, then $G$ is a DLDS-graph.
\end{claim}
\begin{proof}
Suppose that $\delta(G) = 1$. Let $v$ be a support vertex of $G$, and let $L_v$ be the set of leaf neighbors of $v$ in $G$. We now consider the graph $G' = G - (L_v \cup \{v\})$. By our earlier assumptions, the graph $G$ has no dominating vertex. The graph $G'$ is therefore an isolate-free graph. Applying the inductive hypothesis to $G'$, there is a DLDS-partition $\{D',L'\}$ of $V(G')$ into a dominating set $D'$ and a locating dominating set $L'$ of $G'$. As the vertex~$v \in D$ is located by the subset $L_v \subset L$, letting $D = D' \cup \{v\}$ and $L = L' \cup L_v$, we infer that $\{D,L\}$ is a DLDS-partition of $G$. 
\end{proof}

By Claim~\ref{claim1}, we may assume that $\delta(G) \ge 2$, for otherwise $G$ is a DLDS-graph, as desired.

\begin{claim}
\label{claim2}
If $G$ has open twins, then $G$ is a DLDS-graph.
\end{claim}

\begin{proof}
Suppose that $G$ has open twins, and consider a set $X$ of (mutual) open twins in $G$ where each vertex in $X$ has the same open neighborhood, say $Y$. Thus the set $X$ is an independent set, $|X| \ge 2$, and $N_G(x) = Y$ for every vertex $x \in X$. Let $|X| = r \ge 2$ and let $|Y| = s$. Since $\delta(G) \ge 2$, we note that $s \ge 2$. Let $x$ be an arbitrary vertex in $X$ and let $y$ be an arbitrary vertex in~$Y$.

We now consider the graph $G' = G - \{x,y\}$. If $G'$ contains an isolated vertex~$v$, then since the edges between $X$ and $Y$ in $G$ induce a complete bipartite graph $K_{r,s}$ where $r, s \ge 2$, we infer that $v$ is a vertex of degree~$1$ in $G$ with~$y$ as its unique neighbor, contradicting our earlier assumption that $\delta(G) \ge 2$. Hence, $G'$ is an isolate-free graph. Applying the inductive hypothesis to the isolate-free graph $G'$, there is a DLDS-partition $\{D',L'\}$ of $V(G')$ into a dominating set $D'$ and a locating dominating set $L'$ of $G'$.

Let $X' = X \setminus \{x\}$ and let $Y' = Y \setminus \{y\}$. If $D' \cap X' \ne \emptyset$ and $D' \cap Y' \ne \emptyset$, then letting $D = D'$ and $L = L' \cup \{x,y\}$, we infer that $\{D,L\}$ is a DLDS-partition of $G$. Hence we may assume that $D' \cap X' = \emptyset$ or $D' \cap Y' = \emptyset$. Suppose that $D' \cap Y' = \emptyset$, and so $Y' \subseteq L'$. Since $D'$ is a dominating set of $G'$ and $X'$ is an independent set, in order to dominate the vertices in $X'$ we note that $X' \subseteq D'$. In this case, letting $D = D' \cup \{y\}$ and $L = L' \cup \{x\}$, we infer that $\{D,L\}$ is a DLDS-partition of $G$ noting that the vertex $x \in L$ locates the vertex~$y \in D$ since $y$ is the only vertex in $D$ that is adjacent to~$x$.

Hence, we may assume that $D' \cap X' = \emptyset$, for otherwise the desired result holds. Thus, $X' \subseteq L'$. Since $D'$ is a dominating set of $G'$, in order to dominate the vertices in $X'$ we note that $D' \cap Y' \ne \emptyset$. In particular, we note that the vertex~$x$ is adjacent in $G$ to a vertex of $D'$. Suppose that $y$ is adjacent in $G$ to at least one vertex that belongs to the set $D'$. In this case, we let $D = D'$ and $L = L' \cup \{x,y\}$. The resulting partition $\{D,L\}$ is a DLDS-partition of $G$, yielding the desired result.

Hence we may assume that $y$ is adjacent in $G$ to no vertex that belongs to the set $D'$, that is, all neighbors of $y$ in $G$ different from~$x$ belong to the set $L'$. In this case, we let $D = D' \cup \{x\}$ and $L = L' \cup \{y\}$. We note that all vertices in $D'$ are located by the set $L'$. Further, each vertex in $D' \cap Y'$ is adjacent to all vertices in $X'$, where we recall that $X' \subseteq L'$. Since the vertex~$y \in L$ has no neighbor in the set $D$ except for the vertex~$x$, the vertex~$x \in D$ is located by the vertex~$y \in L$ in the set $L$. We therefore infer that the partition $\{D,L\}$ is a DLDS-partition of $G$, once again yielding the desired result that $G$ is a DLDS-graph.
\end{proof}

Notice that the previous proof gives us a way to find a desired partition when there are open twins in a graph which play a special role in locating dominating sets. By Claim~\ref{claim2}, we may also assume that $G$ has no open twins, for otherwise $G$ is a DLDS-graph as desired. Let $C$ be a maximal clique in $G$, and so $C$ is a set of vertices in $G$ that induce a complete graph and no superset of $C$ is a clique. By our earlier assumption, the graph $G$ has no dominating vertex, implying that $C \ne V(G)$. Let $I$ be the set of isolated vertices in $G - C$. Since $G$ is isolate-free, we note that every vertex in $I$ has all its neighbors in $C$, that is,  $N_G(u) \subseteq C$ for all $u \in I$. Moreover since $G$ has no open twins, we note that if $u$ and $v$ are two distinct vertices in $I$, then $N_G(u) \cap C \ne N_G(v) \cap C$.

\begin{claim}
\label{claim3}
If $V(G) = C \cup I$, then $G$ is a DLDS-graph.
\end{claim}

\begin{proof}
Suppose that $V(G) = C \cup I$. Let $D$ be a subset of vertices in the clique $C$ of minimum cardinality that dominates the set $I$. Thus, every vertex in $I$ is adjacent in $G$ to at least one vertex in $D$ and no proper subset of $D$ dominates the set $I$ in the graph $G$. Since $G$ has no dominating vertex, we infer that $|D| \ge 2$. By our choice of the set $D$, if $v$ is an arbitrary vertex in $D$, then $D \setminus \{v\}$ does not dominate the set $I$, implying that there is at least one vertex $v' \in I$ such that $N_G(v') \cap D = \{v\}$. Letting $L = V(G) \setminus D$, the partition $\{D,L\}$ is a DLDS-partition of $G$ by our earlier observation that each vertex in $D$ dominates at least one vertex in $I \subseteq L$ that is not dominated by any other vertex of $D$, yielding the desired result.
\end{proof}

By Claim~\ref{claim3}, we may assume that $C \cup I$ is a proper subset of $V(G)$, for otherwise $G$ is a DLDS-graph as desired. We now consider the graph $G' = G - (C \cup I)$. By definition of the set $I$ we note that the graph $G'$ is isolate-free. Applying the inductive hypothesis to $G'$, there is a DLDS-partition $\{D',L'\}$ of $V(G')$ into a dominating set $D'$ and a locating dominating set $L'$ of $G'$.

\begin{claim}
\label{claim4}
If $I \ne \emptyset$, then $G$ is a DLDS-graph.
\end{claim}

\begin{proof}
Suppose that $I \ne \emptyset$. Let $D_C$ be a subset of vertices in the clique $C$ of minimum cardinality that dominates the set $I$ in $G$. We note that possibly $|D_C| = 1$. We now let $D = D' \cup D_C$ and we let $L = V(G) \setminus D$. We note that $L' \subset L$ and $I \subset L$. Further, we note as in the proof of Claim~\ref{claim3} that by the minimality of the set $D_C$, each vertex in $D_C$ dominates at least one vertex in $I$ that is not dominated by any other vertex of $D_C$. Thus, every vertex in $D_C$ is located by the set $L$. We therefore infer that the partition $\{D,L\}$ is a DLDS-partition of $G$, yielding the desired result.
\end{proof}

By Claim~\ref{claim4}, we may assume that $I = \emptyset$, for otherwise $G$ is a DLDS-graph as desired. Suppose that every vertex in the clique $C$ is adjacent in $G$ to at least one vertex that belongs to the set $D'$. In this case, we let $D = D'$ and $L = L' \cup C$, and we infer that $\{D,L\}$ is a DLDS-partition of $G$. Hence we may assume that at least one vertex, say~$v$, in the clique $C$ has no neighbor in $G$ that belongs to the set $D'$. Thus, $N_G(v) \cap V(G') \subseteq L'$. Let $w$ be an arbitrary vertex in $C \setminus \{v\}$, and let $D = D' \cup \{w\}$ and $L = (L' \cup C) \setminus \{w\}$. In particular, we note that $v \in L$ and every neighbor of $v$ in $G$ belongs to the set $L$, except for the vertex~$w$. The vertex $w$ is therefore located by the vertex~$v \in L$, and we infer that the partition $\{D,L\}$ is a DLDS-partition of $G$. This completes the proof of Theorem~\ref{thm:main}.

\section{Polynomial-time algorithm to construct a DLDS-partition} \label{sec:consequences}

It can be verified that the proofs of Claims~\ref{claim3} and~\ref{claim4} in Section~\ref{S:main-proof} hold even if we replace ``minimum dominating set'' in these proofs by ``\emph{minimal} dominating set'', and the latter can be found in polynomial-time in the order of the input graph. This implies that our proof of Theorem~\ref{thm:main} also provides a polynomial-time algorithm, say $\calA$, to construct a DLDS-partition of any isolate-free graph $G$. To see this, we consider the following six cases depending on $G$ and using the assumptions of Claims~\ref{claim0}, \ref{claim1}, \ref{claim2} and~\ref{claim3}.
\begin{description}
\item[Case 0]: $G$ is one of the graphs of the base cases of the induction.
\item[Case 1]: $G$ has a dominating vertex.
\item[Case 2]: $\delta(G) = 1$.
\item[Case 3]: $G$ has open twins.
\item[Case 4]: $V(G) = C \cup I$, where $C$ is a maximal clique in $G$ and $I$ is the set of isolated vertices in $G - C$. In other words, $G$ is a split graph.
\item[Case 5]: None of the above cases hold.
\end{description}
Notice that each of the above cases can be verified in time $n^{O(1)}$, where $n = |V(G)|$. Moreover, we call $G$ ``terminating'' if either of the following three conditions is true: (i) Case 0 holds; (ii) Case~1 holds; and (iii) Case~4 holds but Cases~2 and~3 do not hold. 

We now describe the algorithm $\calA$ recursively the following way. The terminating condition for $\calA$ is when the graph $G$ itself is terminating. In such a case, $\calA$ constructs a DLDS-partition of $G$ in either constant time (if $G$ is terminating due to Case 0) or in time $n^{O(1)}$ by the proof of either Claim~\ref{claim0} or~\ref{claim3} (if $G$ is terminating due to either Case 1 or Case 4, respectively). On the other hand, if $G$ is not terminating, then $\calA$ considers whichever of the Cases~2,~3 and~5 holds with the smallest case number and constructs a smaller ``inductive'' graph $G'$ by the proof of its corresponding claim, or, for Case~5, by the paragraph following Claim~\ref{claim3}. Then $\calA$ calls itself recursively on $G'$ and proceeds until it terminates. Moreover, for each such recursive call on $G'$ resulting from one of the Cases~2,~3 and~5, notice that a DLDS-partition of the precedent graph $G$ is constructed in time $n^{O(1)}$ by the proof of the corresponding claim or, for Case~5, either by the proof of Claim~\ref{claim4} or by the last paragraph of Section~\ref{S:main-proof}. Since each recursive call of $\calA$ on a graph $G'$ happens for the smallest case number of all the cases that hold for $G'$, the total number of recursive calls during its execution is $O(n)$. This implies that $\calA$ constructs a DLDS-partition of $G$ in polynomial-time in $n$.

\subsection*{Acknowledgements}

This research was supported by the French government IDEX-ISITE initiative 16-IDEX-0001 (CAP 20-25), International Research Center “Innovation Transportation and Production Systems” of the I-SITE CAP 20-25, by the ANR project GRALMECO (ANR-21-CE48-0004). The research of Dipayan Chakraborty was additionally supported in part by the Lebanese American University under the President's Intramural Research Fund PIRF0056. The research of Michael A. Henning was supported in part by the South African National Research Foundation (grants 132588) and the University of Johannesburg. Research of Tero Laihonen was partially supported by the Research Council of Finland grant number 338797. In addition, we thank the anonymous referees for their helpful comments and insightful suggestions, which were valuable additions to this paper.

\end{document}